\documentclass{article}
\usepackage{amsmath}
\usepackage[colorlinks,
            plainpages=false,
            linkcolor=red,
            anchorcolor=blue,
            citecolor=green
            ]{hyperref}
\usepackage{amsthm}
\usepackage{amsfonts}
\usepackage{amssymb}
\usepackage{bbm}
\usepackage{booktabs}
\usepackage{graphicx}
\usepackage{geometry}
\usepackage{xcolor}
\usepackage{indentfirst}
\usepackage[misc]{ifsym}
 \usepackage{cite}
 \usepackage{bm}
\newtheorem{theorem}{Theorem}[section]
\newtheorem{lemma}{Lemma}[section]
\newtheorem{proposition}{Proposition}[section]

\newtheorem{remark}{Remark}[section]
\newtheorem{assumption}{Assumption}[section]
\newtheorem{example}{Example}

\newcommand{\Ass}[1]{\textbf{\upshape A#1}}

\def\d{{\,\rm{d}}}
\begin{document}


\title{ Sharp error estimate of  variable time-step IMEX BDF2 scheme for parabolic integro-differential equations with  initial singularity  arising in finance}

\author{Chengchao Zhao
	\thanks{Beijing Computational Science Research Center, Beijing, 100193, P.R. China (cheng chaozhao@csrc.ac.cn).}
	\and Ruoyu Yang
	\thanks{School of Mathematics and Statistics, Wuhan University, Wuhan 430072, China (ruoyuyang@whu.edu.cn)}
\and Yana Di \thanks {Research Center for Mathematics, Beijing Normal University, Zhuhai 519087, China
Division of Science and Technology, BNU-HKBU United International College, Zhuhai 519087, China (yndi@uic.edu.hk)}
	\and Jiwei Zhang
	\thanks{School of Mathematics and Statistics, and Hubei Key Laboratory of Computational Science, Wuhan University, Wuhan 430072, China (jiweizhang@whu.edu.cn). He is supported partially by NSFC under grant 11771035.}
	}

\maketitle
\begin{abstract}

%

 The recently developed technique of DOC kernels has been a great success in the stability and convergence analysis for BDF2 scheme with variable time steps. However, such an analysis technique seems not directly applicable to  problems with initial singularity. In the numerical simulations of solutions with initial singularity, variable time-steps schemes like the graded mesh are always adopted  to achieve the optimal convergence, whose first adjacent time-step ratio may become pretty large so that the acquired restriction is not satisfied. In this paper, we revisit the variable time-step implicit-explicit two-step backward differentiation formula (IMEX BDF2) scheme presented in [W. Wang, Y. Chen and H. Fang, \emph{SIAM J. Numer. Anal.}, 57 (2019), pp. 1289-1317] to compute the  partial integro-differential equations (PIDEs) with  initial singularity. We obtain the sharp error estimate under a mild restriction condition of adjacent time-step ratios $r_{k}: =\tau_{k}/\tau_{k-1}  \; (k\geq 3)  < r_{\max} = 4.8645 $ and a much mild requirement on the first ratio, i.e.,  $r_2>0$. This leads to the validation of our analysis  of the variable time-step IMEX BDF2 scheme when the initial singularity is dealt by a simple strategy, i.e.,  the graded mesh $t_k=T(k/N)^{\gamma}$.  In this situation, the  convergence of order $\mathcal{O}(N^{-\min\{2,\gamma \alpha\}})$ is achieved with  $N$ and $\alpha$ respectively representing the total mesh points and indicating the regularity of the exact solution. This is, the optical convergence will be achieved by taking $\gamma_{\text{opt}}=2/\alpha$. Numerical examples are provided to demonstrate our theoretical analysis.

    \par\textbf{keywords:} implicit-explicit  method ;  two-step backward differentiation formula; the discrete orthogonal convolution kernels; the discrete complementary convolution kernels; error estimates; variable time-step;
\end{abstract}

\section{Introduction}

In this paper, we consider the computation of   parabolic integro-differential equations (PIDEs) which arise in option pricing theory when the underlying asset follows a jump diffusion process \cite{AA00,M76,W98,KTK15,2019On}
    \begin{equation}\label{continuous}
    \begin{aligned}
	    \partial_{t}u(x,t) - c_1   u_{xx}(x,t) + c_2 u_x(x,t) + c_3 u + \mathcal{J}(u(x,t)) &= f(x,t), &&(x,t) \in \Omega\times (0,T], \\
u(x,t) &= u_b(x,t), &&(x,t) \in \partial\Omega\times (0,T],\\
u(x,0) &=u_0(x), &&x\in \Omega,
    \end{aligned}
    \end{equation}
    where $\Omega = (x_l,x_r)$ with its boundary $\partial \Omega$, the parameters $c_1>0$, $c_2,c_3\in \mathbb{R}$. Here $\mathcal{J}(\cdot)$  represents the nonlocal integral operator and is defined by
    $$\mathcal{J}(u) :=\int_\Omega u(z,t) \rho(x-z)\d z,$$ where $\rho: \mathbb{R}\to \mathbb{R}^+$ is a given function satisfying $\|\rho(x)\|_{L^\infty}\leq C$ for some positive constant $C$. In this situation, there exists a constant $\hat{C}_\rho$ (only depends on the given function $\rho$ and  $\Omega$) such that
     \begin{equation}\label{EQ_CJc}
       \| \mathcal{J}(u) \|_{L^2(\Omega)} \leq \hat{C}_\rho \| u \|_{L^2(\Omega)}.
     \end{equation}
The bound \eqref{EQ_CJc} is satisfied in many practical problems such as  the  finite activity jump diffusion model (e.g., Merton model and Kou model),  CGMY and KoBoL  with infinite
activity and finite variation \cite{C02,C04}.

There are two main considerations in numerically solving PIDEs \eqref{continuous}. The first one   is the discretization of nonlocal integral operators. An undue discretization of the nonlocal integral operator with finite difference schemes,  such as the fully implicit time stepping scheme,  may increase the computational complexity since  inversions will need  solve the resulting systems with full matrices at each time step. To deal with the full matrix, many methods are designed  such as iterative methods \cite{MPS04,AO05,TR00,HFV05}, FFT  \cite{AA00,HFV05} and the  alternating direction implicit (ADI)  method \cite{AA00}. An alternative approach is to avoid  the inversion of a full matrix. To do so, an increasingly popular alternative is the implicit-explicit (IMEX) method \cite{KL11,KL11b,ST14,2019On,KTK15}, which typically treats the nonlocal integral term  explicitly and the rest part implicitly. The IMEX method leads to a tridiagonal system and can be solved  efficiently.

A feature of PIDEs \eqref{continuous}  is the weak singularity near the initial time $t = 0$ arose from the nonsmooth  initial data. For example, as noted in \cite{2019On}, with a nonsmooth  payoff function in option models, the regularity of   exact solution $u(x,t)$ may have the form   $\|  \partial_t^k u \|_{L^2}\leq C t^{\frac12-k}, k=1,2,3$. To be more general, in this paper, we  consider the regularity assumption as follows:
\begin{assumption}\label{assumpation}
There exists a constant $\bar{C}$ such that the solution to \eqref{continuous} satisfies
       \begin{equation}\label{time_regularity}
     	 \|  \partial_t^k u  \| \leq \bar{C} t^{ -k + \alpha },\ t \in \left( 0,T \right],\ k = 1,2,3,
       \end{equation}
where $\partial_t^k u := \frac{ \partial^k u }{ \partial t^k } $ and the regularity parameter $\alpha$ satisfies $\frac12\leq \alpha\leq 1$.
\end{assumption}

For solving such  initial singularity problems, another consideration is how to develop efficient and accurate algorithms.  A heuristic method to improve efficiency without sacrificing accuracy is the adaptive time-step scheme. Thus, one can employ small time steps  when the dynamics evolves rapidly, and use large time steps when the dynamics evolves slowly \cite{2021beam,2020crystal,2019On}. Another method is to employ high-order schemes in time to have the same accuracy with a relatively large time-step, for example, the  Runge--Kutta methods \cite{BNR07}, Crank--Nicolson method \cite{HFV05,STS14} and two-step backward
differentiation formula (BDF2) \cite{AO05,KTK15,2019On}.  Specifically, the BDF2 method  has received much attention due to its strong stability  (A-stable) \cite{unbounded,2021beam,2019On,LTZ20,LSTZ21,2020crystal,2020reaction,SFF18}.

The numerical analysis of BDF2 scheme with variable time steps receives much attention. For instance,  Becker \cite{B98} (also Thom\'ee's classical book \cite[Lemma 10.6]{Galerkin})  presents the stability and convergence with a factor $\exp(C\Gamma_n)$ under the adjacent time-step ratio satisfying $0< r_k\leq (2+\sqrt{13})/3\approx 1.868$, where $\Gamma_n := \sum_{k=2}^{n-2} \max\{0,r_k-r_{k+2}\}$.
Emmrich \cite{E05} improves the Becker's condition up  to $ 1.91$.  By using a novel generalized discrete Gr\"onwall-type inequality, Chen et al. \cite{unbounded} consider the Cahn-Hilliard equation, and present the energy stability  without the factor $\Gamma_n$ with $0< r_k\leq 3.561$.  The works \cite{2019diffusion,2020reaction} consider linear parabolic equations based on DOC kernels with $0\leq r_k \leq 3.561$ \cite{2019diffusion} and $0<r_k\leq 4.8645$ \cite{2020reaction}, respectively. There are also a great progress on the error estimates for nonlinear equations, see  \cite{2021beam,LSTZ21,2020crystal,LTZ20}.

It is worthy to point out that the analysis in literature mentioned above is based on smooth solutions, and will fail for the initial singularity \eqref{time_regularity}. In fact, an established technique to restore an optimal convergence rate is to employ a graded mesh
\begin{equation} \label{gm}
t_k=T(k/N)^\gamma \quad \text{for} \; \; 0\leq k\leq N,
\end{equation}
where the parameter $\gamma$ is used to adapt to the strength of the singularity.  Choosing $\gamma = 1$ will lead to a uniform mesh, and the larger the value of $\gamma$ the more  strongly the temporal meshes are concentrated at the initial time.  For instance, such meshes have long been applied to numerically solve Fredholm \cite{GG82} and Volterra \cite{HB85} integral equations, and time-fractional PDEs \cite{SOG17,LLZ18,LMZ20,LMZ21}.

Given the typical regularity \eqref{time_regularity} with $\alpha = 1/2$, Wang et al. \cite{2019On} consider the sharp error estimate of a variable-time-step IMEX-BDF2 scheme  for the problem \eqref{continuous}.  A half-order convergence is presented under $0 < r_k  \leq 1.91$, and the second-order convergence is achieved by introducing two kinds of nonuniform grids. For the graded mesh \eqref{gm},  the the maximum error for the fully BDF2 scheme  will be order of
\begin{equation} \label{OC}
N^{-\min\{2,\gamma\alpha\}},
\end{equation}
where we ignore the additional error due to the spatial discretization.
From \eqref{OC}, if $\gamma \geq 2/\alpha$ then the error is $N^{-2}.$ Specifically,  for the case of $\alpha = 1/2$,  one needs to take $\gamma \geq 4$ to have the optimal convergence of $N^{-2}.$  In this situation,  the first-level adjacent time-step ratio defined in \eqref{mr} satisfies  $r_2:=\tau_2/\tau_1=2^{\gamma}-1\geq 15$, which significantly breaks the adjacent time-step ratio restriction given in \cite{2019On,2019diffusion,2020reaction} at the first-step ratio. This is to say, the analysis in \cite{2019On,2019diffusion,2020reaction} will be not valid any longer when the initial regularity is considered. It is numerically indicated in \cite{N19} that the error at the first step shall become larger as the first ratio $r_2$ increases. Thus, it is desired to develop  a more general analytic framework and find out the effect of the first ratio  on the error bound.

The aim of this paper is to revisit the IMEX-BDF2 scheme with variable time steps given in \cite{2019On}  for solving the PIDE \eqref{continuous}, and achieve the sharp error estimate  under the following adjacent time-step ratios condition
\begin{center}
\Ass{1} : \quad $r_2>0$ and $0< r_k \leq r_{\max}-\delta,\quad 3\leq k \leq N$ for any small constant $0<\delta < r_{\max}$.
\end{center}

In \Ass{1}, we point out that  $r_{\max} = \frac{1}{6} \left(\sqrt[3]{1196\!-\!12\sqrt{177}}+\sqrt[3]{1196\!+\!12\sqrt{177}}\right)+\frac43 \approx 4.8645 $ is the root of $x^3 = (2x+1)^2$, and the parameter $\delta$ can be taken any small value in the practical adaptive time-step strategies.
Our contributions can be listed as follows.
\begin{itemize}
\item An alternative representation of $\Ass{1}$ is given in \cite{LJWZ21}  by $0<r_k\leq r_{\text{user}}<r_{\max}$, where $r_{\text{user}}$ can be considered as $r_{\text{user}} = r_{\max}-\delta$ in this paper. We introduce $\delta$ aimed to clearly indicate how do the stability and convergence estimates  depend  on the gap $\delta = r_{\max}-r_{\text{user}}$.
\item Comparing the recent work \cite{2019On}, we extend to  the adjacent time-step ratios condition from $0 < r_k  \leq 1.91$ to $0 < r_k  \leq 4.86$. More importantly, our ratio condition in $\Ass{1}$ has no restriction on the first ratio $r_2$.  It implies that our analysis will remain valid to deal with the initial singularity by using the graded mesh for any $\gamma \leq 2/\alpha$. This is,  an optimal convergence of order $\mathcal{O}(N^{-2})$ is achieved under the graded mesh  $t_k=T(k/N)^{\gamma_{\text{opt}}}$ with an optimal parameter $\gamma_{\text{opt}}=2/\alpha$.
\item Besides, our error estimates clarify that the error increases polynomially with respect to the first ratio $r_2$, which is consistent with the numerical experiments in \cite{N19}.
\end{itemize}

Our stability analysis and sharp error estimate for the IMEX-BDF2 scheme  are based on the  recently developed DOC \cite{2019diffusion} and DCC  \cite{2020reaction} kernels by developing their new properties,  which is used to overcome the difficulties arising from the initial singularity, including a strictly positive definiteness of the BDF2 kernels $b^{(n)}_{n-k}$  (see Lemma \ref{thpositive}) and a sharper estimate of  DCC kernels $p^{(n)}_{n-k}$   (see Proposition \ref{pttau}). In addition, the finite difference method is used for spatial discretization. Differing from the finite element method and spectral method, the finite difference method may break the discrete integration by parts formula, which
 may bring extra difficulties in the analysis.  A novel convolution-type Young's inequality (see Lemma \ref{Lemma_Dn}) is developed to handle the spatial difficulties.

The remainder is organized as follows. Section \ref{Section2} presents IMEX BDF2 scheme, and several properties of DOC and DCC kernels including some new properties such as the strictly positive definiteness of the BDF2 kernels $b^{(n)}_{n-k}$  and a sharper estimate of  the  kernels. The stability and half-order convergence under the ratio restriction \Ass{1} are presented in Section \ref{Section3} and the optimal second-order convergence are also achieved under a graded mesh with  an optimal parameter in Section \ref{Section4}. Numerical experiments are provided to demonstrate our theoretical analysis in Section \ref{Section5}.

   \section{IMEX BDF2 scheme, DOC and DOC kernels \label{Section2}}

   In this paper, we consider an IMEX-BDF2 scheme with variable time steps combining with the finite difference method in space. We first take the generally variable time grids by $0=t_{0}<t_{1}<\cdots<t_{N}=T$, and denote the $k$th time-step size by $\tau_{k}:=t_{k}-t_{k-1}$, the maximum time step size by $\tau:=\max_{1\leq k\leq N}\tau_{k}$, and the adjacent time-step ratio by
     \begin{equation} \label{mr}
       r_{k} = \frac {\tau_{k}}{\tau_{k-1}},\quad 2 \leq k \leq N,\qquad r_{1} = 0.
     \end{equation}
Denote $u^{n}(x)$ by the approximation of  $u(x,t_{n})$, and $\nabla_{\tau}u^{n} = u^{n} - u^{n-1}$ by  the difference operator. The BDF1 and BDF2 with variable time steps are respectively defined by
     \begin{equation}\notag
       \mathcal{D}_{1}u^{n} = \frac {1}{\tau_{n}}\nabla_{\tau}u^{n},\qquad \mathcal{D}_{2}u^{n} = \frac {1 + 2r_{n}}{\tau_{n}(1 + r_{n})}\nabla_{\tau}u^{n} - \frac {r_{n}^{2}}{\tau_{n}(1 + r_{n})}\nabla_{\tau}u^{n-1},\quad \textrm{for}\quad 2 \leq n \leq N.
     \end{equation}
Since the BDF2 needs two starting values,  the BDF1 is used to compute the first-level solution $u^1$ and  the BDF2 is used when $n\geq 2$.  If we introduce the notations, i.e.,  $b_{0}^{(1)} := 1/\tau_{1}$ and
     \begin{equation}
       b_{0}^{(n)} = \frac {1 + 2r_{n}}{\tau_{n}(1 + r_{n})}, \quad b_{1}^{(n)} = - \frac {r_{n}^{2}}{\tau_{n}(1 + r_{n})}\quad \textrm{and}\quad b_{j}^{(n)} = 0, \quad\textrm{for}\ 2\leq j \leq n - 1,\label{B}
     \end{equation}
BDF2 (using BDF1 to compute  $u^1$) can be written in the following discrete convolution form
 \begin{equation}
       \mathcal{D}_{2}u^n := \sum_{k=1}^{n} b_{n-k}^{(n)}\nabla_{\tau}u^{k},\quad n \geq 1.
     \end{equation}

The finite difference method is considered for the spatial discretization. We denote the spatial length $h = (x_r-x_l)/M$ for a positive integer $M$, the discrete grid $\bar{\Omega}_h =\{ x_l + ih|0 \leq i \leq M\},  \Omega_h = \bar{\Omega}_h\cap \Omega$,
and $\partial \Omega_h  = \bar{\Omega}_h\cup \partial \Omega$. For any grid function $v_h = \{v_i | v_i = v(x_i), x_i\in \bar{\Omega}_h\}$, let $\Delta_h v_i$ be the standard second-order approximation of $v_{xx}(x_i)$ and $\nabla_h v_i:=\frac{v_{i+1}-v_{i-1}}{2h}$ be the approximation of $v_{x}(x_i)$. Denote the space of
grid functions by $\mathcal{V}_h = \{v_h| v_h \text{\ vanishes on\ } \partial \Omega_h\}$.
The discrete $L^2$ inner product and associated discrete $L^2$-norm are defined respectively by
$\langle u,v\rangle  := \sum_{i=1}^{M-1} h u_iv_i,\;  \|u\| := \sqrt{\langle u,u\rangle}, \;  \forall u,v\in \mathcal{V}_h.
$ The discrete $H^1$ semi-norm and $H^1$-norm are defined respectively by $|u|_1 =  \sqrt{\langle-\Delta_h u,u\rangle},\; \|u\|_1 = \sqrt{\|u\|^2+|u|_1^2},\; \forall u\in \mathcal{V}_h.
$

 The integral operator in  \eqref{continuous} is approximated by the composite trapezoidal rule
\begin{equation*}\begin{split}
	\mathcal{J}(u_i^n) = \int_\Omega u(z,t_n)\rho(z-x_i) \d z \approx \frac{h}{2}\left( u_0^n \rho_{i,0}^n + 2\sum_{j=1}^{M-1} u_j^n \rho_{i,j}^n + u_M^n \rho_{i,M}^n \right):=\mathcal{J}_h(u_i^n),
\end{split}\end{equation*}
where $\rho_{i,j}^n = \rho( x_j-x_i , t_n)$.  Similar to \eqref{EQ_CJc}, there exists constant  $C_\rho$ such that
\begin{align}
\|\mathcal{J}_h (u^n_h)\|\leq C_J \|u^n_h\|.\label{EQ_CJ}
\end{align}
Thus, a fully discrete IMEX BDF2 scheme for solving the  PIDE \eqref{continuous} is given by
     \begin{equation}\label{IMEXsp}
     \begin{aligned}
       \mathcal{D}_{2}u^{n}_i - c_1 \Delta_h u^{n}_i + c_2 \nabla_h u^{n}_i + c_3 u^{n}_i + \mathcal{J}_h(Eu^{n-1}_i) &= f^{n}_i,&& x_i\in \Omega_h, 1 \leq n \leq N,\\
       u^n_i&=u_b(x_i,t_n), &&x_i\in \partial\Omega_h, 1 \leq n \leq N,\\
       u^0_i&=u_0(x_i), && x_i\in \Omega_h,
     \end{aligned}
     \end{equation}
     where $f^n:=f(t_n)$, $Eu^{n-1} = (1+r_{n})u^{n-1}-r_{n}u^{n-2}$ for $n\geq 2$ and $Eu^0 = u^0$.

   \subsection{The definitions and properties of DCC and DOC kernels}
The techniques of the DCC and DOC kernels will play a key role in the following stability and convergence analysis. Here we first introduce the DCC and DOC kernels \cite{2020reaction,2019diffusion,2021beam}, and also develop a new positive definiteness of the DOC kernels and a sharp estimate of the DCC kernels.

     The discrete complementary convolution (DCC) kernels $p_{n-j}^{(n)}$ \cite{2020reaction} is defined by
          \begin{equation}
       \sum_{j=k}^{n}p_{n-j}^{(n)}b_{j-k}^{(j)} \equiv 1,\quad \forall 1 \leq k \leq n,\,1 \leq n \leq N,\label{PB}
     \end{equation}
which has the property of
     \begin{equation}
       \sum_{j=1}^{n}p_{n-j}^{(n)}\mathcal{D}_{2}u^{j} = \sum_{j=1}^{n}p_{n-j}^{(n)}\sum_{l=1}^{j}b_{j-l}^{(j)}\nabla_{\tau}u^{l}
       = \sum_{l=1}^{n}\nabla_{\tau}u^{l}\sum_{j=l}^{n}p_{n-j}^{(n)}b_{j-l}^{(j)} = u^{n}-u^{0},\quad \forall n\geq 1.\label{PD}
     \end{equation}
Denote by $\delta_{nk}$ the Kronecker delta symbol with $\delta_{nk} = 1$ if $n = k$ and $\delta_{nk} = 0$ if $n \neq k$. The discrete orthogonal convolution (DOC) kernels \cite{2019diffusion}  is defined by
     \begin{equation}
       \sum_{j=k}^{n}\theta_{n-j}^{(n)}b_{j-k}^{(j)} \equiv \delta_{nk},\quad \forall 1 \leq k \leq n,\label{THB}
     \end{equation}
which has the property of
     \begin{equation}
       \sum_{j=1}^{k}\theta_{k-j}^{(k)}\mathcal{D}_{2}u^{j} = \sum_{l=1}^{k}\nabla_{\tau}u^{l}\sum_{j=l}^{k}\theta_{k-j}^{(k)}b_{j-l}^{(j)} = u^{k}-u^{k-1},\quad 1 \leq k \leq N. \label{THD}
     \end{equation}
As noted in \cite[Proposition 2.1]{2020reaction}, the DCC and DOC kernels have the following relations
     \begin{align}
     &\theta_{0}^{(n)} = p_{0}^{(n)},\; \theta_{n-k}^{(n)} = p_{n-k}^{(n)} - p_{n-k-1}^{(n-1)} \; (1 \leq k \leq n-1) \quad \text{and} \quad  p_{n-j}^{(n)} = \sum_{k=j}^{n}\theta_{k-j}^{(k)}\;(1 \leq j \leq n). \label{PTH}
     \end{align}

The DOC and DCC kernels may be calculated explicitly in the following two lemmas.
     \begin{lemma}\label{EQ_theta}\rm{\cite[Lemma\ 2.3]{2019diffusion}}
       The DOC kernels $\theta_{n-j}^{(n)}$ have the following properties:
       \begin{align}
         &\theta_{n-j}^{(n)} > 0,\quad \forall \ 1 \geq j \geq n \qquad \text{and}\qquad  \sum_{j=1}^{n}\theta_{n-j}^{(n)} = \tau_{n},\qquad \text{for}\ n \geq 1.\label{EQ_1633}
       \end{align}

     \end{lemma}
   \begin{proposition}\label{pttau}\rm{\cite[Proposition 2.2]{2020reaction}}
   The DCC kernels $p_{n-k}^{(n)}$ defined in \eqref{PB} satisfy
       \begin{align}
         &p_{n-j}^{(n)} = \sum_{k=j}^{n} \frac {\tau_{k}(1+r_{j})}{1+2r_{j}} \prod_{i=j+1}^{k} \frac {r_{i}}{1+2r_{i}}\leq 2\tau,\quad 1 \leq j \leq n \qquad \text{and} \qquad \sum_{j=1}^{n} p_{n-j}^{(n)} = t_{n},\label{EQ_psum}
       \end{align}
       where $\prod_{i=j+1}^{k} = 1\ \text{for}\ j \geq k$ is defined.
     \end{proposition}
%
We point out that the BDF2 kernels $b^{(n)}_{n-k}(k\geq2)$ still keep  positively defined for any ratio $r_2 >0$ in \Ass{1}. The similar result for the version of $r_2<4.8645$ is presented in \cite{LJWZ21}.
\begin{lemma}\label{bpositive}
 	Assume the time step ratio $r_{k}$ satisfies $\mathbf{A1}$. For any real sequence $\{\omega_{k}\}_{k=1}^{n}$, it holds that
 		\begin{align}
 			&2\omega_{k}\sum_{j=2}^{k}b_{k-j}^{(k)}\omega_{j} \geq \frac {r_{k+1}\sqrt{r_{\max}}}{1+r_{k+1}}\frac {\omega_{k}^{2}}{\tau_{k}} - \frac {r_{k}\sqrt{r_{\max}}}{1+r_{k}}\frac {\omega_{k-1}^{2}}{\tau_{k-1}} + C_r\delta \frac{\omega_{k}^{2}}{\tau_{k}},\quad k \geq 3,\label{bbb}\\
 			&2\sum_{k=2}^{n}\omega_{k}\sum_{j=2}^{k}b_{k-j}^{(k)}\omega_{j} \geq \sum_{k=2}^{n}C_r\delta \frac{\omega_{k}^{2}}{\tau_{k}},\quad n \geq 2 ,
 		\end{align}
 		where $C_{r} =\sqrt{r_{\max}}/(1+r_{\max})^{2}$ and $\delta$ is any small constant satisfying $0 < \delta < r_{\max}$ (see $\mathbf{A1}$).
 	
 	\begin{proof}
 The inequality \eqref{bbb} can be obtained by \cite[Lemma 3.2]{2021beam}.
Summing \eqref{bbb} from 2 to $n$ yields
 		\begin{align*}
 			2\sum_{k=2}^{n}\omega_{k}\sum_{j=2}^{k}b_{k-j}^{(k)}\omega_{j}
 			&\geq \frac{ 2( 1 + 2r_2 ) }{ 1 + r_2 } \frac{ \omega_2^2}{ \tau_2 } + \frac{ r_{n+1} \sqrt{r_{\max}}}{ 1 + r_{n+1} } \frac{ \omega_n^2 }{ \tau_{n} } - \frac{ r_3\sqrt{r_{\max}} }{ 1 + r_3 } \frac{ \omega_2^2 }{  \tau_2 }
 			+ \sum_{k=3}^{n}\frac {\delta\sqrt{r_{\max}}\omega_{k}^{2}}{(1+r_{\max})^{2}\tau_{k}}\\
 			&\geq \left( 2 - \frac{ r_{\max}^{3/2} } { ( 1 + r_{\max} )  } \right) \frac{ \omega_2^2 }{ \tau_2 }
 			+ \sum_{k=2}^{n}\frac {\delta\sqrt{r_{\max}}\omega_{k}^{2}}{(1+r_{\max})^{2}\tau_{k}}\\
 			&\geq \frac{\omega_{2}^{2}}{(1+r_{\max})\tau_{2}}+\sum_{k=2}^{n}\frac {\delta\sqrt{r_{\max}}\omega_{k}^{2}}{(1+r_{\max})^{2}\tau_{k}}\geq \sum_{k=2}^{n}\frac {\delta\sqrt{r_{\max}}\omega_{k}^{2}}{(1+r_{\max})^{2}\tau_{k}},
 		\end{align*}
 where one uses the facts $r_{\max}^{3/2}=1+2r_{\max}$ and $0<\delta<r_{\max}$.
 		The proof is completed.  \end{proof}
 \end{lemma}
\subsection{The new properties of DOC kernels and DCC kernels}


    \begin{lemma}\label{thpositive}
    	Assume the condition $\mathbf{A1}$ holds. Then the DOC kernels $\theta_{n-k}^{(n)}( n \geq 2)$ defined in \eqref{THB} are  positive definite. Moreover, for any real sequences $\{\omega_{j}\}_{j=1}^{n}$, it holds that
    	\begin{equation}\label{theta_posi_express}
    		2\sum_{k=2}^{n}\omega_{k}\sum_{j=2}^{k}\theta_{k-j}^{(k)}\omega_{j} \geq C_r \delta \sum_{k=2}^{n}\tau_{k}^{-1}(\sum\limits_{j=2}^{k}\theta_{k-j}^{(k)}\omega_{j})^{2},\quad \forall n \geq 2,
    	\end{equation}
	 where $C_{r} = \sqrt{r_{\max}}/(1+r_{\max})^{2}$.
    \end{lemma}

\begin{proof}
       For any $\bm{\omega} = (\omega_2,\dots,\omega_n) \in \mathbb{R}^n$, we construct the vector $\bm{v} = (v_2,\dots,v_n) \in \mathbb{R}^n$ by
       \begin{equation}\label{vnomega}
       	\omega_j = \sum_{l=2}^{j} b_{j-l}^{(j)} v_l,\ 2 \leq j \leq n.
       \end{equation}
It is easy to check $\bm{v}$ is uniquely determined by $\bm{\omega}$.
  Multiplying $\theta^{(k)}_{k-j}$ on both sides of \eqref{vnomega} and summing $j$ from 2 to $k$, one finds
       \begin{align}
       \sum_{j=2}^k \theta^{(k)}_{k-j} \omega_j = \sum_{j=2}^k\theta^{(k)}_{k-j}\sum_{l=2}^{j} b_{j-l}^{(j)} v_l=\sum_{l=2}^kv_l\sum_{j=l}^k\theta^{(k)}_{k-j} b_{j-l}^{(j)} =v_k,\label{EQ_06061510}
       \end{align}
       where one exchanges the order of summations and uses the definition \eqref{THB}.
       Again, multiplying   \eqref{EQ_06061510} by $\omega_{k}$ and taking summation from 2 to n, we find
       \begin{equation*}
       	 2\sum_{k=2}^{n}\omega_{k} \sum_{j=2}^{k} \theta_{k-j}^{(k)} \omega_{j}
       	 = 2\sum_{k=2}^{n} \omega_{k} v_{k}
       	 = 2\sum_{k=2}^{n} v_{k} \sum_{j=2}^{k} b_{k-j}^{(k)} v_{j}.
       \end{equation*}
Thus, the proof is completed by      Lemma \ref{bpositive} and \eqref{EQ_06061510}.
     \end{proof}

The following proposition gives a sharp estimate of the DCC kernels, which play a key role in the stability and convergence analysis.
 \begin{proposition}\label{Pro_p}
      Assume \Ass{1} holds and $c_r = r_{\max}^{5/2}$. The DCC kernels $p_{n-k}^{(n)}$ defined in \eqref{PB} satisfy
       \begin{align}
       	 p_{n-j}^{(n)} &\leq c_r\delta^{-1} \sqrt{\tau_j}\sqrt{\tau},\qquad  2\leq j \leq n,\label{EQ_1500}\\
       p_{n-1}^{(n)} &\leq \tau_1+c_r\delta^{-1} \sqrt{\tau_2}\sqrt{\tau}.\label{EQ_1528}
       \end{align}
     \end{proposition}
\begin{proof}
From Proposition \ref{pttau} and the condition \Ass{1}, the DCC kernels $p^{(n)}_{n-j}$ can be bounded by
       \begin{align}
       	 p_{n-j}^{(n)} &= \sum_{k=j}^n \frac{\tau_k (1+r_j) }{1+2r_j} \prod_{i=j+1}^{k} \frac {r_{i}}{1+2r_{i}} = \sqrt{\tau_j} \sum_{k=j}^n \sqrt{\tau_k} \frac{1+r_j}{1+2r_j} \prod_{i=j+1}^{k} \frac {r_{i}^{3/2}}{1+2r_{i}}\nonumber\\
       	 &\leq \sqrt{\tau_j} \sum_{k=j}^n \sqrt{\tau_k} \prod_{i=j+1}^{k} \frac {( r_{\max} - \delta )^{3/2}}{1+2(r_{\max} - \delta)}\nonumber\\
       	 &\leq \sqrt{\tau_j} \sum_{k=j}^n \sqrt{\tau_k} \left( \frac { ( r_{\max} - \delta )^{3/2} } { 1+2(r_{\max} - \delta) } \right)^{k-j}\nonumber\\
       	 &\leq \sqrt{\tau_j\tau}\left(1-\frac { ( r_{\max} - \delta )^{3/2} } { 1+2(r_{\max} - \delta) }\right)^{-1}, \quad \text{for}\ 2\leq j \leq n.
       \label{EQ_ppp}
       \end{align}
     Noting that $(r_{\max})^{3/2} = 1+2r_{\max}$, we get
       \begin{align}
       1-\frac { ( r_{\max} - \delta )^{3/2} } { 1+2(r_{\max} - \delta) }&=\frac { ( r_{\max})^{3/2} } { 1+2r_{\max}}-\frac { ( r_{\max} - \delta )^{3/2} } { 1+2(r_{\max} - \delta) }
       \geq  \frac{\sqrt{r_{\max}}}{(1+2r_{\max})^2}\delta=r_{\max}^{-5/2}\delta.\label{EQ_cr}
       \end{align}
Thus, by inserting \eqref{EQ_cr} into \eqref{EQ_ppp}, we obtain \eqref{EQ_1500}.

For $j=1$, the estimate is different because of $r_2\in \mathbb{R}^+$. In this situation, from the definition we have the following estimate
\begin{align}
p^{(n)}_{n-1} &= \sum_{k=1}^n \tau_k \prod_{i=2}^{k} \frac {r_{i}}{1+2r_{i}} =\tau_1 + \frac{r_2}{1+2r_2}\sum_{k=2}^n \frac{\tau_k (1+r_2) }{1+2r_2} \prod_{i=3}^{k} \frac {r_{i}}{1+2r_{i}}=\tau_1+\frac{r_2}{1+2r_2}p^{(n)}_{n-2}.\label{EQ_1515}
\end{align}
The proof is completed by inserting the estimate \eqref{EQ_1500} into \eqref{EQ_1515}.
\end{proof}

\section{Stability and convergence analysis for IMEX BDF2 scheme\label{Section3}}

 We now consider the stability and  convergence analysis for the IMEX BDF2 scheme \eqref{IMEXsp}, which requires a discrete Gr\"{o}nwall inequality given as follows.
     \begin{lemma}\label{gronwall}
       Assume that $\lambda >$ 0 and the sequences $\{v_j\}_{j=1}^N$ and $\{\eta_j\}_{j=1}^N$ are nonnegative. If
       \begin{equation*}
         v_n \leq \lambda \sum_{j = 1}^{n-1} {\tau_j v_j} + \sum_{j=0}^{n} {\eta_j},\quad \text{for}\quad 1 \leq n \leq N,
       \end{equation*}
       then it holds that
       \begin{equation*}
         v_n \leq \exp \ ( \lambda t_{n-1} ) \sum_{j = 0}^n {\eta_j},\quad \text{for}\quad 1 \leq n \leq N.
       \end{equation*}
       \end{lemma}

   The proof of Lemma \ref{gronwall} can be derived by the standard induction hypothesis and we omit it here.

   \subsection{Stability }
Let $u^n_h$ and $\hat{u}^n_h$ be solutions of the IMEX BDF2 scheme \eqref{IMEXsp} with  initial values $u^0_h, \hat{u}^0_h$ and source terms $f^n_h, \hat{f}^n_h$, respectively.  Let $\phi_h^n:=u^n_h-\hat{u}^n_h,  \zeta^n_h:=f^n_h-\hat{f}^n_h(0\leq n\leq N)$ be the perturbations. Then $\phi_h^n$ solves
     \begin{equation}\label{EQ_Pequation}
       \mathcal{D}_{2}\phi_h^n - c_1 \Delta_h \phi_h^{n} + c_2 \nabla_h \phi_h^{n} + c_3 \phi_h^n + \mathcal{J}(E\phi_h^{n-1}) = \zeta^n_h,\qquad 1 \leq n \leq N,
     \end{equation}
     with the initial value $\phi_h^0(x)$ and   homogeneous Dirichlet boundary condition.  Different from the spectral method and finite element method, the finite difference method may bring  extra difficulties in the analysis. For example,  the discrete integration by parts formula that $\langle -\Delta_h u_h, v_h\rangle = \langle -\nabla_h u_h, \nabla_h v_h\rangle$ is not satisfied due to the incompatible discretizations of $u_{xx}$ and $u_x$. The following lemmas will play a key role in  overcoming this difficulty.
\begin{lemma}\label{Lemma_D2n}
It holds for  any $\epsilon>0$ and $u_h,v_h\in \mathcal{V}_h$ that
\begin{align}
2\langle -\Delta_h u_h,u_h-v_h \rangle&\geq |u_h|_1^2-|v_h|_1^2, \label{EQ_0229}\\
|u_h|_1&\geq \|\nabla_h u_h\|, \label{EQ_0242}\\
|u_h|_1^2 +\epsilon^2\|u_h\|^2&\geq 2\epsilon  \langle\nabla_h u_h, u_h\rangle.\label{EQ_1501}
\end{align}
\end{lemma}
\begin{proof}
The first claim \eqref{EQ_0229} holds by the inequality $2a(a-b) \geq a^2-b^2$, i.e.,
\begin{align*}
2\langle -\Delta_h u_h,u_h-v_h \rangle&= \frac2{h}\sum_{i=1}^M (u_i-u_{i-1})\big((u_i-u_{i-1})-(v_i-v_{i-1})\big)\nonumber\\
&\geq  \frac1{h}\sum_{i=1}^M \big((u_i-u_{i-1})^2-(v_i-v_{i-1})^2\big)=|u|_1^2-|v|_1^2.\end{align*}
And the third claim \eqref{EQ_1501} can be immediately derived by the second claim \eqref{EQ_0242} via the Young's inequality.
Hence, we only focus on proving \eqref{EQ_0242}.
From the definition of discrete semi-norm, we have
\begin{align*}
 |u|_1^2  &= \frac1{h}\sum_{i=1}^{M-1}(2u_i-u_{i+1}-u_{i-1})u_i\\
&=\frac1{2h}\sum_{i=1}^{M-1}\left((u_{i+1}-u_i)^2+(u_{i}-u_{i-1})^2\right)
+\frac{u_1^2+u_{M-1}^2}{2h}\\
&=\frac1{4h}\sum_{i=1}^{M-1}\left((u_{i+1}-u_{i-1})^2+(u_{i+1}-2u_i+u_{i-1})^2\right)+\frac{u_1^2+u_{M-1}^2}{2h}\\
&=\|\nabla_h u_h\|^2+\frac{h^2}{4}\|\Delta_h u_i\|^2+\frac{u_1^2+u_{M-1}^2}{2h},
\end{align*}
where the identity $2a^2+2b^2 = (a+b)^2+(a-b)^2$ is used. The proof is completed.
\end{proof}
\begin{lemma}\label{Lemma_Dn}
Assume $\mathbf{A1}$ holds. Then it holds for  any $\epsilon>0$ and $u_h^k\in \mathcal{V}_h \;(1\leq k\leq n)$ that
\begin{align}
2\sum_{k=2}^{n}\sum_{j=2}^k \theta^{(k)}_{k-j}\langle-\Delta_h u_h^j,u_h^k\rangle+ \frac{\epsilon^2}{C_r\delta}\sum_{k=2}^n \tau_k \|u^k\|^2&\geq2\epsilon \sum_{k=2}^{n}\sum_{j=2}^k \theta^{(k)}_{k-j}\langle\nabla_h u^j, u^k\rangle,\label{EQ_1402}
\end{align}
where  $C_r$ is defined in \eqref{theta_posi_express}.
\end{lemma}
\begin{proof}
On the one hand, from the definition of discrete Laplacian operator $\Delta_h$, we have
\begin{align*}
&\quad\langle-\Delta_h u_h^j,u_h^k\rangle  = \frac1{h}\sum_{i=1}^{M-1}(2u_i^j-u_{i+1}^j-u_{i-1}^j)u_i^k\\
&=\frac1{2h}\sum_{i=1}^{M-1}\left((u^j_{i+1}-u^j_i)(u^k_{i+1}-u^k_{i})+(u^j_{i}-u^j_{i-1})(u^k_{i}-u^k_{i-1})\right)
+\frac{u^j_1u^k_1+u^j_{M-1}u^k_{M-1}}{2h}.
\end{align*}
From Lemma \ref{thpositive}, one further has
\begin{align}
&\quad 2\sum_{k=2}^{n}\sum_{j=2}^k \theta^{(k)}_{k-j}\langle-\Delta_h u_h^j,u_h^k\rangle\nonumber\\
&\geq \frac{C_r\delta}{2h}\sum_{i=1}^{M-1}\sum_{k=2}^n\frac1{\tau_k} \left( \Big(\sum_{j=2}^k \theta_{k-j}^{(k)}(u_{i+1}^j-u_i^j)\Big)^2+ \Big(\sum_{j=2}^k \theta_{k-j}^{(k)}(u_{i}^j-u_{i-1}^j)\Big)^2\right).\label{EQ_1525}
\end{align}
On the other hand,
\begin{align}
&\quad 2\epsilon\sum_{k=2}^{n}\sum_{j=2}^k \theta^{(k)}_{k-j} \langle\nabla_h u^j, u^k\rangle \\
&= \epsilon\sum_{k=2}^{n}\sum_{j=2}^k \theta^{(k)}_{k-j}\sum_{i=1}^{M-1} \left((u_{i+1}^j-u^j_i)+(u_i^j-u^j_{i-1})\right)u^k_i\nonumber\\
&= \epsilon\sum_{i=1}^{M-1}\sum_{k=2}^{n}\left(\sum_{j=2}^k \theta^{(k)}_{k-j}\Big((u_{i+1}^j-u^j_i)+(u_i^j-u^j_{i-1})\Big)\right)u^k_i\nonumber\\
&\leq  \frac{C_r\delta}{2h}\sum_{i=1}^{M-1}\sum_{k=2}^n\frac1{\tau_k} \left( \Big(\sum_{j=2}^k \theta_{k-j}^{(k)}(u_{i+1}^j-u_i^j)\Big)^2+ \Big(\sum_{j=2}^k \theta_{k-j}^{(k)}(u_{i}^j-u_{i-1}^j)\Big)^2\right)\nonumber\\
&\quad+\frac{\epsilon^2}{C_r\delta}\sum_{i=1}^{M-1}h\sum_{k=2}^n \tau_k (u_i^k)^2.\label{EQ_1526}
\end{align}
Thus, combining \eqref{EQ_1525} with \eqref{EQ_1526}, we have \eqref{EQ_1402}. The inequality \eqref{EQ_1501} can be  proven similarly and is omitted here. The proof is completed.
\end{proof}

     We now present the result of stability.
  \begin{theorem}\label{prioritheorem}
       Assume the  condition $\mathbf{A1}$ holds. Then  the IMEX BDF2 scheme \eqref{IMEXsp} is unconditionally stable in the $L^2$-norm. This is, if the maximum time-step size satisfies
        \begin{align}
\tau\leq \min\left\{\frac1{2C_1},\frac1{2({c_2}^2 / {c_1} + 4|c_3| + 2C_J )},\frac1{4(5C_J+4|c_3|)}\right\},\label{EQ_2055}
\end{align}
it holds  for $n\geq 2$ that
  \begin{align}
  \|\phi^{n}_h\|&\leq 2\exp(2C_1t_{n-1})\Big(\big(3+r_2\big)\|\phi_h^0\| + 7\sum_{k=1}^{n}  p_{n-k}^{(n)}\|\zeta^k_h \| +6p^{(n)}_{n-2}C_J(\tau_1+\tau_2)\|\zeta^1_h\|+C_2\sqrt{\tau}|\phi_h^0|_1\Big)\label{EQ_stability}\\
     &\leq 2\exp(2C_1t_{n-1})\Big(\big(3+r_2\big)\|\phi_h^0\|+7t_n\max_{1\leq k\leq n}\|\zeta^k_h\|+6p^{(n)}_{n-2}C_J(\tau_1+\tau_2)\|\zeta^1_h\|+C_2\sqrt{\tau}|\phi_h^0|_1\Big),\\
         \|\phi^1_h\| &\leq 2 \|\phi^0_h\| + 3 t_1 \|\zeta^1_h\|,\label{EQ_1441}
       \end{align}
       where
       \begin{align*}
C_1 :=  \frac{{c_2}^2}{ (c_1 C_r \delta)} +2|c_3|+2C_J ( 1 +2 r_{\max} ), \; C_2:=\frac{4c_r}{\delta}\sqrt{ c_1r_2}. \end{align*}
     \end{theorem}
     \begin{proof}
For  $n \geq 2$, it follows from \eqref{EQ_Pequation} and \eqref{THD} that
\begin{align}
\sum_{j=2}^k \theta_{k-j}^{(k)} \mathcal{D}_2 \phi_h^j = \sum_{j=1}^k \theta_{k-j}^{(k)} \mathcal{D}_2 \phi_h^j-\theta_{k-1}^{(k)}b^{(1)}_0 \nabla_\tau \phi_h^1=\nabla_\tau \phi_h^k-\theta_{k-1}^{(k)}b^{(1)}_0\nabla_\tau \phi_h^1.\label{EQ1522}
\end{align}
Applying the identity  \eqref{EQ1522} to the  perturbed equation \eqref{EQ_Pequation}, one has
       \begin{equation}\label{usumjk}
         \nabla_\tau \phi_h^k -\theta_{k-1}^{(k)}b_0^{(1)}\nabla_\tau \phi_h^1 -  c_1 \sum_{j=2}^k \theta_{k-j}^{(k)}\Delta_h \phi^j_h = - \sum_{j=2}^{k} \theta_{k-j}^{(k)} \left ( c_2 \nabla_h \phi^j_h + c_3 \phi_h^j + \mathcal{J}(E\phi_h^{j-1}) - \zeta^j_h \right ).
       \end{equation}
       Taking inner products with $2\phi_h^k$ on both sides of \eqref{usumjk} and summing the resulting from 2 to $n$, we get
       \begin{align}\label{EQ_06081420}
        &\quad  2\sum_{k=2}^n\langle \nabla_\tau \phi_h^k,\phi_h^k\rangle - 2 c_1\sum_{k=2}^n\sum_{j=2}^k \theta_{k-j}^{(k)} \langle\Delta_h \phi^j_h,\phi_h^k \rangle
          \notag\\
         &= - 2\sum_{k=2}^n\sum_{j=2}^{k} \theta_{k-j}^{(k)}\langle c_2 \nabla_h \phi^j_h + c_3 \phi_h^j + \mathcal{J}(E\phi_h^{j-1}) - \zeta^j_h, \phi_h^k \rangle+ 2\sum_{k=2}^n \theta_{k-1}^{(k)} \langle b_0^{(1)} \nabla_\tau \phi_h^1, \phi_h^k \rangle.
       \end{align}
Taking $\epsilon = c_2/c_1$ in \eqref{EQ_1402} and inserting the resulting into  \eqref{EQ_06081420},  one has
       \begin{equation}\label{usum2}
       \begin{aligned}
         2\sum_{k=2}^n \left\langle \nabla_\tau \phi_h^k , \phi_h^k \right\rangle\leq& \frac{ {c_2}^2 }{ c_1 C_r \delta} \sum_{k=2}^n\tau_k \left\| \phi_h^k \right\|^2 +2\sum_{k=2}^n \left\langle \theta_{k-1}^{(k)}b_0^{(1)} \nabla_\tau \phi_h^1 , \phi_h^k \right\rangle\\
         &+ 2\sum_{k=2}^n \left\langle  \sum_{j=2}^{k} \theta_{k-j}^{(k)} \left ( c_3 \phi_h^j + \mathcal{J}(E\phi_h^{j-1}) - \zeta^j_h \right ) , \phi_h^k \right\rangle.
         \end{aligned}
       \end{equation}
Applying $2a(a-b)\geq a^2-b^2$,  \eqref{EQ_CJ} and the Cauchy-Schwarz inequality to \eqref{usum2}, one has
       \begin{align}
         \|\phi^n_h\|^2 - \|\phi^1_h\|^2 &\leq \frac{ {c_2}^2 }{ c_1 C_r \delta} \sum_{k=2}^n\tau_k \left\| \phi^k_h \right\|^2 + 2|c_3| \sum_{k=2}^n\sum_{j=2}^k \theta_{k-j}^{(k)}\|\phi^j_h\|\|\phi^k_h\| + 2\sum_{k=2}^n\|\sum_{j=2}^k \theta_{k-j}^{(k)}\zeta^j_h\| \|\phi^k_h\| \notag\\
         &+ 2C_J \sum_{k=3}^n\sum_{j=3}^k \theta_{k-j}^{(k)} \bigg( ( 1 + r_{\max} ) \|\phi^{j-1}_h\| + r_{\max} \|\phi^{j-2}_h\| \bigg) \|\phi^k_h\|\\
         &+ 2C_J \sum_{k=2}^n \theta_{k-2}^{(k)} \bigg( ( 1 + r_2 ) \|\phi^{1}_h\| +r_2 \|\phi^{0}_h\| \bigg) \|\phi^k_h\| + \frac{2}{\tau_1} \sum_{k=2}^n \theta_{k-1}^{(k)}  \|\nabla_\tau \phi^{1}_h\| \|\phi^k_h\|.\notag
       \end{align}
       Selecting an integer $n_0$ ($0 \leq n_0 \leq n$) such that $\|\phi^{n_0}\| = \underset{0 \leq k \leq n}{{\max}} \| \phi^k \|,$ then we have
       \begin{equation}\begin{split}
         \|\phi^{n_0}_h\|^2 \leq &\|\phi^1_h\|\|\phi^{n_0}_h\| + \frac{ {c_2}^2 }{ c_1 C_r \delta} \|\phi^{n_0}_h\| \sum_{k=2}^{n_0}\tau_k \left\| \phi^k_h \right\| + 2|c_3| \|\phi^{n_0}_h\|\sum_{k=2}^{n_0} \|\phi^k_h\| \sum_{j=2}^k \theta_{k-j}^{(k)}\\
         &+ 2C_J ( 1 +2 r_{\max} )\|\phi^{n_0}_h\|\sum_{k=3}^{n_0} \|\phi^k_h\| \sum_{j=3}^k \theta_{k-j}^{(k)} + 2\|\phi^{n_0}_h\|\sum_{k=2}^{n_0} \|\sum_{j=2}^k \theta_{k-j}^{(k)}\zeta^j_h\|\\
         &+ 2C_J \|\phi^{n_0}_h\|\sum_{k=2}^{n_0} \theta_{k-2}^{(k)} \left( (1 + r_2 )\| \phi^1_h \| + r_2 \| \phi^0_h \| \right) + \frac{2}{\tau_1} \|\phi^{n_0}_h\|\sum_{k=2}^{n_0} \theta_{k-1}^{(k)}\| \nabla_\tau \phi^1_h \|.\label{EQ_2323}
       \end{split}\end{equation}
       Eliminating a $\|\phi^{n_0}_h\|$ from both sides, and setting $
C_1 := \frac{ {c_2}^2 }{ c_1 C_r \delta} +2|c_3|+2C_J ( 1 +2 r_{\max} ),
$ one arrives at
              \begin{equation}
              \begin{split}
         \|\phi^{n}_h\|\leq &\|\phi^1_h\| + C_1 \sum_{k=2}^{n}\tau_k \left\| \phi^k_h \right\|  + 2\sum_{k=2}^{n}  p_{n-k}^{(n)}\|\zeta^k_h\|\\
         &+ 2C_Jp^{(n)}_{n-2} \left( (1 + r_2 )\| \phi^1_h \| + r_2 \| \phi^0_h \| \right) + \frac{2}{\tau_1} (p^{(n)}_{n-1}-\tau_1)\| \nabla_\tau \phi^1_h \|,\label{EQ_2324}
       \end{split}\end{equation}
       where one exchanges the order of summation, uses the property \eqref{PTH} and the facts $\phi_h^{n_0}\geq \phi_h^k(0\leq k\leq n)$ and $n_0\leq n$. It is necessary  to bound the terms $\|\phi_h^1\|$ and $\|\nabla_\tau \phi_h^1\|$. For the first term $\|\phi_h^1\|$,
setting $n=1$ in \eqref{EQ_Pequation} and   taking inner products with $2\phi_h^k$ on both sides, one has
       \begin{equation}
       	2\left\langle \nabla_\tau \phi_h^1 , \phi_h^1 \right\rangle + 2c_1 \tau_1\langle-\Delta_h \phi^1_h, \phi^1_h\rangle
       	= - 2c_2\tau_1 \left\langle \nabla_h \phi^1_h , \phi_h^1 \right\rangle - 2c_3 \tau_1 \| \phi_h^1 \|^2 - 2\tau_1 \left\langle \mathcal{J}(\phi_h^0)+\zeta^1_h , \phi_h^1 \right\rangle.\label{EQ_1954}
       \end{equation}
       Applying the property \eqref{EQ_CJ}, the identity $2a(a-b) = a^2-b^2+(a-b)^2$, the inequality \eqref{EQ_1501}  and the Cauchy-Schwarz inequality to \eqref{EQ_1954}, we obtain
       \begin{equation}
       	\|\phi_h^1\|^2 - \|\phi_h^0\|^2+\|\nabla_\tau \phi_h^1\|^2 \leq
       	(\frac{  {c_2}^2 }{ 2c_1 } +2 |c_3|)\tau_1 \left\| \phi_h^1 \right\|^2  + 2 C_J \tau_1 \|\phi_h^0\| \|\phi_h^1\| + 2 \tau_1 \|\zeta^1_h\| \|\phi_h^1\|.\label{EQ_2328}
       \end{equation}
       Applying the Young's inequality  again  to \eqref{EQ_2328}, one has
              \begin{equation}
       \left(\frac34-(\frac{ {c_2}^2 }{2 c_1 } +2 |c_3|+C_J)\tau_1\right) 	\|\phi_h^1\|^2  \leq
       	  (1+C_J\tau_1)\|\phi_h^0\|^2 +  4\tau_1^2 \|\zeta^1_h\|^2.
       \end{equation}
              Taking $\tau \leq 1/ (2{c_2}^2 / {c_1} + 8|c_3| + 4C_J )$, we finally get
                     \begin{equation}\label{EQ_1424}
	\|\phi_h^1\|^2  \leq \frac52\|\phi_h^0\|^2 +  8\tau_1^2 \|\zeta^1_h\|^2.
       \end{equation}
Noting that $\sqrt{a^2+b^2}\leq |a|+|b|$, we further have
                     \begin{equation}\label{EQ_1435}
	\|\phi_h^1\|  \leq 2\|\phi_h^0\| +  3\tau_1 \|\zeta^1_h\|.
       \end{equation}
We now estimate the second term $\frac{\|\nabla_\tau \phi_h^1\|}{\tau_1}$.
Setting $n=1$ in \eqref{EQ_Pequation} and   taking inner products on both sides   with $2\nabla_\tau \phi_h^1$, one has
       \begin{align}
       &\quad	\frac{2\| \nabla_\tau \phi_h^1\|^2}{\tau_1} - 2 c_1 \left\langle \Delta_h \phi^1_h ,\nabla_\tau \phi_h^1 \right\rangle\nonumber\\
       	&= - 2c_2 \left\langle \nabla_h \phi^1_h ,\nabla_\tau \phi_h^1 \right\rangle - 2c_3 \left\langle \phi_h^1,\nabla_\tau \phi_h^1\right \rangle - 2 \left\langle \mathcal{J}(\phi_h^0) , \nabla_\tau\phi_h^1 \right\rangle -2 \left\langle \zeta^1_h , \nabla_\tau\phi_h^1 \right\rangle.\label{EQ_1415}
       \end{align}
Hence, applying \eqref{EQ_CJ}, \eqref{EQ_0229} and the Cauchy-Schwarz inequality to \eqref{EQ_1415}, we obtain
\begin{align}
\frac{2\|\nabla_\tau \phi_h^1\|^2}{\tau_1}+c_1 |\phi^1_h|_1^2 \leq c_1 |\phi^0_h|_1^2+2\left(c_2\|\nabla_h \phi^1_h\|+|c_3|\|\phi_h^1\|+C_J\|\phi_h^0\|+\|\zeta^1_h\|\right)\|\nabla_\tau \phi_h^1\|.\label{EQ_1421}
\end{align}
Then applying  the inequality \eqref{EQ_0242} and the Young's inequality  to \eqref{EQ_1421}, one  yields
\begin{align}
\frac{\|\nabla_\tau \phi_h^1\|^2}{\tau_1}\leq c_1 |\phi^0_h|_1^2+ \frac{c_2^2}{c_1}\|\nabla_\tau\phi_h^1\|^2+\tau_1\left(3c_3^2\|\phi_h^1\|^2+3C_J^2\|\phi_h^0\|^2+3\|\zeta^1_h\|^2\right).
\end{align}
Taking $\tau\leq  1/{2({c_2}^2 / {c_1} + 4|c_3| + 2C_J )}\leq  c_1/(2c_2^2)$,    we arrive at
\begin{align}\label{EQ_1426}
\frac{\|\nabla_\tau \phi_h^1\|^2}{\tau_1}\leq 2c_1 |\phi^0_h|_1^2+2\tau_1\left(3c_3^2\|\phi_h^1\|^2+3C_J^2\|\phi_h^0\|^2+3\|\zeta^1_h\|^2\right).
\end{align}

Inserting the inequality \eqref{EQ_1424} into \eqref{EQ_1426}, one has
\begin{align*}
\|\nabla_\tau \phi_h^1\|^2\leq 2c_1\tau_1|\phi^0_h|_1^2+3\tau_1^2(5c_3^2+2C_J^2)\|\phi_h^0\|^2+6\tau_1^2(8c_3^2\tau_1^2+1)\|\zeta^1_h\|^2,
\end{align*}
which together with the inequality $\sqrt{a^2+b^2}\leq |a|+|b|$ implies
\begin{align}
\|\nabla_\tau \phi_h^1\|\leq 2\sqrt{c_1\tau_1}|\phi^0_h|_1+\tau_1(4|c_3|+3C_J)\|\phi_h^0\|+\tau_1(7|c_3|\tau_1+3)\|\zeta^1_h\|.\label{EQ_1434}
\end{align}
Inserting the estimates \eqref{EQ_1435} and \eqref{EQ_1434} into \eqref{EQ_2324}, one produces
\begin{align*}
 \|\phi^{n}_h\| \;&\leq2\|\phi^0_h\| +3\tau_1\|\zeta^1_h\|+2C_1 \sum_{k=2}^{n}\tau_k \left\| \phi^k_h \right\| + 2\sum_{k=2}^{n}  p_{n-k}^{(n)}\|\zeta^k_h\|\\
&+2\left(3p^{(n)}_{n-2}C_J(\tau_1+\tau_2)+(p^{(n)}_{n-1}-\tau_1)(7|c_3|\tau_1+3)\right)\|\zeta^1_h\|\\
 &+2p^{(n)}_{n-2}\left((5+3r_2)C_J+4|c_3|\right)\|\phi^0_h\|+C_2\sqrt{\tau}|\phi_h^0|_1,
 \end{align*}
where one uses Lemma \ref{EQ_theta} and Proposition \ref{Pro_p} and identity \eqref{EQ_1525}.
It follows from \eqref{EQ_2055} that
 \begin{equation}
 \begin{split}
 \|\phi^{n}_h\|\leq &(6+2r_2)\|\phi^0_h\| + 2C_1 \sum_{k=2}^{n-1}\tau_k \left\| \phi^k_h \right\|+ 14\sum_{k=1}^{n}  p_{n-k}^{(n)}\|\zeta^k_h\|\\
 &+12p^{(n)}_{n-2}C_J(\tau_1+\tau_2)\|\zeta^1_h\|+2C_2\sqrt{\tau}|\phi_h^0|_1, \end{split}\end{equation}
where Proposition \ref{Pro_p} and identity \eqref{EQ_1515} are used. Hence, it follows from Lemma \ref{gronwall} that
\begin{align}
\|\phi^{n}_h\|\leq 2\exp(2C_1t_{n-1})&\Big((3+r_2)\|\phi_h^0\| + 7\sum_{k=1}^{n}  p_{n-k}^{(n)}\|\zeta^k_h\| +6p^{(n)}_{n-2}C_J(\tau_1+\tau_2)\|\zeta^1_h\|+C_2\sqrt{\tau}|\phi_h^0|_1\Big).\label{EQ_1442}
\end{align}
According to \eqref{EQ_psum}, we finally arrive at
\begin{align*}
		\|\phi^{n}_h\|&\leq 2\exp(2C_1t_{n-1})\Big((3+r_2)\|\phi_h^0\| + 7\sum_{k=1}^{n}  p_{n-k}^{(n)}\|\zeta^k_h\|+6p^{(n)}_{n-2}C_J(\tau_1+\tau_2)\|\zeta^1_h\|+C_2\sqrt{\tau}|\phi_h^0|_1\Big)\nonumber\\
		&\leq 2\exp(2C_1t_{n-1})\Big((3+r_2)\|\phi_h^0\|+7t_n\max_{1\leq k\leq n}\|\zeta^k_h\|+6p^{(n)}_{n-2}C_J(\tau_1+\tau_2)\|\zeta^1_h\|+C_2\sqrt{\tau}|\phi_h^0|_1\Big).
\end{align*}
The proof is completed.
\end{proof}
\begin{remark}
Although the stability estimate \eqref{EQ_stability} relies on the $H^1$ semi-norm of the initial perturbation, i.e., $|\phi_h^0|_1$, its weight shall tends to zero as $\tau\to 0$. In this sense, the $H^1$ semi-norm $|\phi_h^0|_1$ has mild effect on stability of the IMEX-BDF2 scheme \eqref{IMEXsp}. Actually, the  effect of  $H^1$ semi-norm $|\phi_h^0|_1$ can be completely removed if $r_k<1+\sqrt{2}$ for all $3\leq k\leq N$. More specifically, under such a  severe restriction, the DCC kernels may be further estimated as
\begin{align}
p_{n-j}^{(n)} = \frac{ (1+r_j)\tau_j }{1+2r_j}\sum_{k=j}^n \prod_{i=j+1}^{k} \frac {r_{i}^2}{1+2r_{i}}\leq C\tau_j,\quad 2\leq j\leq n,
\end{align}
where $C$ is a constant independent of ratio $r_2$ and mesh sizes $\tau, h$.
Hence, the last term in \eqref{EQ_2324} can be bounded by $2r_2(\|\phi^1_h\|+\|\phi^0_h\|)$ and the $H^1$ semi-norm $|\phi_h^0|_1$ is removed.
\end{remark}

   \subsection{Consistence and convergence}
   Set $e^n_h = u(t_n, x_h) - u^n_h, x_h\in \Omega_h$. From \eqref{continuous}, the error function satisfies the governing equation
   \begin{equation}
     \mathcal{D}_{2}e^{n}_h - c_1 \Delta_h e^{n}_h + c_2 \nabla_h e^{n}_h + c_3 e^{n}_h + \mathcal{J}_h(Ee^{n-1}_h) = \xi^n_h + \eta^n_h +\nu^n_h,\quad\text{for}\quad 1 \leq n \leq N,
   \end{equation}
   where $\xi^n_h:= \mathcal{J}(Eu(t_{n-1},x_h)) - \mathcal{J}(u(t_n,x_h)), \eta^n:= \mathcal{D}_{2}u(t_n,x_h) - \partial_t u(t_n,x_h)$ with $ x_h\in \mathcal{V}_h$,  and $\nu^n_h$ represents the spatial truncation error. For the uniform spatial mesh, it is known the  spatial truncation error has second-order accuracy. This is,  there exists a constant $C_s$ such that $\|\nu^n_h\|\leq C_s h^2$.

     \begin{lemma}\label{tilded}
     Under the assumption \ref{assumpation}, it holds that
\begin{align}
\|\xi^j_h\| \leq &C_J \bar{C}(\tau_j^2t_{j-1}^{\alpha-2}+\tau_j\tau_{j-1}t_{j-2}^{\alpha-2}),\quad j\geq 3,\label{EQ_2119}\\
\|\xi^2_h\|\leq &C_J\bar{C}r_2\tau_1^\alpha(r_2+1/\alpha),\label{EQ_2122}\\
\|\xi^1_h\|\leq &C_J\bar{C}\tau_1^\alpha/\alpha, \label{EQ_2124}
\end{align}
     \end{lemma}
     \begin{proof}
     By using the Taylor expansion, it is easy to check that the integro error $\xi^j_h:= \mathcal{J}(Eu(t_{j-1},x_h)) - \mathcal{J}(u(t_j,x_h))(x_h\in \mathcal{V}_h, 1\leq j\leq N)$ can be expressed as
     \begin{equation}
       \xi^j_h:= \mathcal{J}(R^j_h),
     \end{equation}
     where
     \begin{equation}\label{dexpress}\begin{split}
       &R^j_h: = \int_{t_{j-1}}^{t_j} ( t - t_j) \partial_{tt} u(t,x_h) \d t + r_j \int_{t_{j-2}}^{t_{j-1}} ( t_{j-2} - t ) \partial_{tt} u(t,x_h)  \d t,\quad 2 \leq j \leq N,\\
       &R^1_h: = u(t_0,x_h) - u(t_1,x_h) = - \int_0^{t_1} \partial_{t} u(t,x_h)  \d t.
     \end{split}\end{equation}
Combining \eqref{EQ_CJ} and assumption \ref{assumpation}, one immediately has \eqref{EQ_2119}--\eqref{EQ_2124}.
The proof is completed.
     \end{proof}

\begin{lemma}\label{truncation}
Under the assumption \ref{assumpation}, for the truncation error $\eta^j$ it holds that
 \begin{align}
 \|\eta^j_h\|  &\leq 2\bar{C}\tau_j^2t_{j-1}^{\alpha-3}+\frac12\bar{C}\tau_{j-1}^2t_{j-2}^{\alpha-3},\quad j\geq 2,\label{EQ_2145} \\
 \|\eta^2_h\|  &\leq ( 2r_2^2+1/(2\alpha))\bar{C}\tau_{1}^{\alpha-1},\label{EQ_2146}\\
 \|\eta^1_h\|  &\leq \frac{\bar{C}}{\alpha}\tau_1^{\alpha-1},\label{EQ_2147}
 \end{align}
\end{lemma}
\begin{proof}
By using the  Taylor expansion \cite{2019diffusion,2020reaction,Galerkin}, the truncation error $\eta^j_h$ may be expressed by
\begin{align*}
\eta^j_h = &-\frac{1+2r_j}{(2+2r_j)\tau_j}\int_{t_{j-1}}^{t_j}(t-t_{j-1})^2\partial_{ttt} u(t,x_h) \d t+\frac{r_j^2}{(2+2r_j)\tau_j}\int_{t_{j-2}}^{t_{j-1}}(t-t_{j-2})^2\partial_{ttt} u(t,x_h) \d t\nonumber\\
&+\frac{r_j}{(2+2r_j)}\int_{t_{j-1}}^{t_j}(2(t-t_{j-1})+\tau_{j-1})\partial_{ttt} u(t,x_h) \d t, \quad j\geq 2,\\
\eta^1_h = &- \frac{1}{\tau_1} \int_0^{t_1} t\partial_{tt}u(t,x_h)dt.
\end{align*}
 Noting the weak singularity of solutions in the assumption \ref{assumpation}, we have
 \begin{align}
 \|\eta^j_h\| \leq &\frac{1+2r_j}{(1+r_j)}\bar{C}\tau_j^2t_{j-1}^{\alpha-3}+\frac{r_j}{(2+2r_j)}\bar{C}\tau_{j-1}^2t_{j-2}^{\alpha-3}\leq 2\bar{C}\tau_j^2t_{j-1}^{\alpha-3}+\frac12\bar{C}\tau_{j-1}^2t_{j-2}^{\alpha-3},\quad j\geq 2, \\
 \|\eta^2_h\| \leq &\frac{1+2r_2}{(1+r_2)}\bar{C}\tau_2^2\tau_1^{\alpha-3}+\frac{r_2}{(2+2r_2)\tau_1}\bar{C}\int_{0}^{t_1}t^{\alpha-1}\d t\leq ( 2r_2^2+1/(2\alpha))\bar{C}\tau_{1}^{\alpha-1},\\
 \|\eta^1_h\| \leq  & \frac{\bar{C}}{\tau_1}\int_0^{t_1}t\partial_{tt}u \d t\leq \frac{\bar{C}}{\alpha}\tau_1^{\alpha-1}.
 \end{align}
 The proof is completed.
\end{proof}

In the remainder of this paper, any subscripted $C$ and $C_u$, denotes
 positive constants, not necessarily the same at different occurrences, which is
always dependent on the given data and the solution, but independent of the ratio $r_2$ and mesh sizes $\tau$ and $h$.

     \begin{theorem}\label{errorestimate1}
       Let $u(t,x)$ be the solution to problem \eqref{continuous} and assume \Ass{1}  holds, then the solution $u^n_h$ to BDF2 scheme \eqref{IMEXsp} is convergent in the $L^2$-norm.
       This is, if the maximum time-step sizes satisfy \eqref{EQ_2055},
it holds  for $\ 3 \leq n \leq N$ that
\begin{align}
      \|e^{n}_h\|&\leq C_u\exp(2C_1t_{n-1})\Big((1+r_2)\|e^0_h\|+C_2\sqrt{\tau}|e^0_h|_1 +\sum_{k=3}^n p^{(n)}_{n-k}(\tau_k^2t_{k-1}^{\alpha-3}+\tau_{k-1}^2t_{k-2}^{\alpha-3}) \nonumber\\
      &\qquad \qquad \qquad \qquad +(1+r_2)^2(p^{(n)}_{n-2}+p^{(n)}_{n-1})+h^2\Big),\label{EQ_1627}\\
      &\leq C_u\exp(2C_1t_{n-1})\Big((1+r_2)\|e^0_h\| +C_2\sqrt{\tau}|e^0_h|_1+\delta^{-1}\sum_{k=3}^n \tau^\alpha(\tau_k^{3-\alpha} t_{k-1}^{\alpha-3}+\tau_{k-1}^{3-\alpha}t_{k-2}^{\alpha-3}) \nonumber\\
      &\qquad \qquad \qquad\qquad+\delta^{-1}(1+r_2)^2(r_2^{1-\alpha}+1)\tau^\alpha+h^2\Big),\label{EQ_1640}\\
      \|e^{2}_h\|&\leq C_u\exp(2C_1t_1)\left((1+r_2)\|e^0_h\| + \tau_1^{\alpha}+t_2h^2+C_2\sqrt{\tau}|e^0_h|_1\right),\\
      \|e^1_h\|&\leq C_u\Big( \|e^0_h\| +\tau_1^\alpha+t_1 h^2\Big),
     \end{align}
where $C_u$  is a constant independent of ratio $r_2$ and mesh sizes $\tau$ and $h$.

   \end{theorem}
     \begin{proof}
For $n\geq2$, it follows from \eqref{EQ_1441} and \eqref{EQ_1442} that
 \begin{align}
\|e^{n}_h\|&\leq 2\exp(2C_1t_{n-1})\Big((3+r_2)\|e^0_h\| +C_2\sqrt{\tau}|e^0_h|_1+ 14\sum_{k=1}^np^{(n)}_{n-k} \|\xi^k_h+\eta^k_h+\nu^k_h\|\nonumber\\
     &\qquad \qquad\qquad\qquad+6p^{(n)}_{n-2}C_J(\tau_1+\tau_2)\|\xi^1_h+\eta^1_h+\nu^1_h\|\Big),\label{EQ_06091439}\\
         \|e^1_h\| &\leq 2 \|e^0_h\| + 3 \tau_1 \|\xi^1_h+\eta^1_h+\nu^1_h\|.\label{e1}
       \end{align}
For the first-level error $e^1_h$, inserting \eqref{EQ_2124}  \eqref{EQ_2147} into  \eqref{e1}, one has
\begin{align}
\|e^1_h\|\leq 2 \|e^0_h\| + 3\bar{C} (C_J\tau_1+1)\tau_1^\alpha/\alpha+3C_s\tau_1 h^2\leq C_u( \|e^0_h\|+\tau_1^\alpha+t_1h^2).
\end{align}
For $n=2$, it follows from  Lemmas \ref{tilded} and \ref{truncation} and Proposition \ref{pttau} that
\begin{align*}
& \quad 7\sum_{k=1}^2 p^{(2)}_{2-k}\|\xi^k_h +\eta^k_h+\nu^k_h\|+6C_Jp^{(2)}_{0}(\tau_1+\tau_2)\|\xi^1_h+\eta^1_h+\nu^1_h\| \\
&\leq C_u\Big((1+r_2)^3\tau_1^{\alpha} +t_2h^2\Big),
\end{align*}
Combining with \eqref{EQ_06091439}, one has
\begin{align}
\|e^{2}_h\|&\leq C_u\exp(2C_1t_1)\left((1+r_2)\|e^0_h\| + \tau_1^{\alpha}+t_2h^2+C_2\sqrt{\tau}|e^0_h|_1\right).
\end{align}
For $n\geq 3$,  from Lemmas \ref{tilded} and \ref{truncation}, one has
 \begin{align}
 &\sum_{k=1}^n p^{(n)}_{n-k}\|\eta^k_h\| \leq 2\bar{C}\Big(\sum_{k=3}^n p^{(n)}_{n-k}(\tau_k^2t_{k-1}^{\alpha-3}+\tau_{k-1}^2t_{k-2}^{\alpha-3})
 +(p^{(n)}_{n-2}( r_2^2+1)+p^{(n)}_{n-1})\tau_1^{\alpha-1}\Big),\label{EQ_1443}\\
 &\sum_{k=1}^n p^{(n)}_{n-k}\|\xi^k_h\| \leq C_J \bar{C}\Big(\sum_{k=3}^n p^{(n)}_{n-k}(\tau_k^2t_{k-1}^{\alpha-2}+\tau_k\tau_{k-1}t_{k-2}^{\alpha-2})+(p^{(n)}_{n-2}r_2(r_2+2)+2p^{(n)}_{n-1})\tau_1^\alpha\Big),\label{EQ_1521}\\
&p^{(n)}_{n-2}(\tau_1+\tau_2)\|\xi^1_h+\eta^1_h\| \leq 2\bar{C}(C_J\tau_1+1)(\tau_1+\tau_2)p^{(n)}_{n-2}\tau_1^{\alpha-1},
 \end{align}
 where $\alpha\geq 1/2$ is used. Hence, \eqref{EQ_1627} holds by inserting \eqref{EQ_1443} and \eqref{EQ_1521} into \eqref{EQ_06091439} and using the condition $\tau\leq 1$.
From Proposition \ref{Pro_p}, we find
\begin{align}
(p^{(n)}_{n-2}+p^{(n)}_{n-1})\tau_1^{\alpha-1} \leq c_r\delta^{-1} (2r_2^{1-\alpha}+1)\tau^\alpha,
\end{align}
which yields \eqref{EQ_1640}.
The proof is completed.  \end{proof}

\begin{remark}
Theorem \ref{errorestimate1} indicates  that the error increases polynomially with respect to the first ratio $r_2$, which is consistent with the numerical experiments in \cite{N19}.
\end{remark}

\section{Graded mesh for time grid\label{Section4}}
An important feature of the considered problem has the weak regularity of the solution arose from the nonsmooth initial data. A fundamental flaw of  numerical scheme with uniform time step will bring the loss of accuracy.  Theorem \ref{errorestimate1} indicates that taking smaller time steps near the initial time $t=0$ may be an effective approach to improve the global accuracy. In this section, we consider the
graded time mesh $t_k = T( k/N )^\gamma$ to achieve the second-order convergence, where  the grading parameter $\gamma > 1$. We claim that the graded time mesh is consistent with the ratio restriction \Ass{1}. Actually, it is easy to verify   the time-step ratio $r_k = \frac{ \tau_k }{ \tau_{k-1} } = \frac{ k^\gamma - ( k-1 )^\gamma }{ ( k-1 )^\gamma - ( k-2 )^\gamma }$ for $2 \leq k \leq N$. Note that the function
$$h(x) = \frac{ x^\gamma - ( x-1 )^\gamma }{ ( x-1 )^\gamma - ( x-2 )^\gamma }$$
 is monotone decreasing. One can verify $r_3\leq r_{\max}\approx4.8645$ for $\gamma \leq 4.2529$, which implies the ration restriction \Ass{2} holds for the graded time mesh.

\begin{theorem}\label{Theorem_Graded}
Assume the conditions in Theorem \ref{errorestimate1} hold and take the graded mesh $t_k = T( k/N )^\gamma$. Then the global error $e^n$ has the following convergence rate:
   \begin{equation}\label{EQ_Graded}
   	 \|e^n\|\leq
   	 \begin{cases}
   	   CN^{- \alpha\gamma},\ &\mbox{ $\gamma < 2/\alpha$ }\\
   	   CN^{-2} \log{N},\ &\mbox{ $\gamma = 2/\alpha$ }\\
   	   CN^{-2},\ &\mbox{ $\gamma > 2/\alpha$ },
     \end{cases}
   \end{equation}
   where $C$ is a constant independent with $N$.
\end{theorem}
\begin{proof}
For the case of $n\leq 2$, it is easy to verify that \eqref{EQ_Graded} holds.  We only prove the case of $n\geq 3$.

It follows from Theorem \ref{errorestimate1} that
\begin{align}
\|e^n\|\leq C\big(\sum_{k=3}^n p^{(n)}_{n-k}(\tau_k^2t_{k-1}^{\alpha-3}+\tau_{k-1}^2t_{k-2}^{\alpha-3}) +(p^{(n)}_{n-2}+p^{(n)}_{n-1})\tau_1^{\alpha-1}+h^2\big).
\end{align}
Due to the decreasing property of function $h(x)$,  there exists  a  positive integer $N_0(\gamma)$ such that $r_k < 1 + \sqrt{2}$ when $k \geq N_0(\gamma)$ (For example, $N_0(\gamma) = 5$ for $\gamma = 4$). Combining with proposition \ref{pttau},  i.e., the DCC kernels satisfy
\begin{equation}
   	 p_{n-j}^{(n)} = \sum_{k=j}^{n} \frac {\tau_{k}(1+r_{j})}{1+2r_{j}} \prod_{i=j+1}^{k} \frac {r_{i}}{1+2r_{i}} \leq \tau_{j} \sum_{k=j}^n \prod_{i=j+1}^k \frac{ r_i^2 }{ 1 + 2 r_i },
   \end{equation}
which implies that there exists a constant $C$ such that $p_{n-j}^{(n)}\leq C\tau_j$. Hence, we have
\begin{align}
\|e^n\|\leq C\big(\sum_{k=3}^n (\tau_k^3t_{k-1}^{\alpha-3}+\tau_{k-1}^3t_{k-2}^{\alpha-3}) +\tau_1^{\alpha}+h^2\big).\label{EQ_1610}
\end{align}
It is easy to verify $\tau_k \leq T \gamma N^{-1} {(k/N)}^{\gamma - 1} $ for $1 \leq k \leq N$,  then we have
\begin{align}
\sum_{k=3}^n \tau_k^3t_{k-1}^{\alpha-3} \leq T^\alpha \gamma^3\sum_{k=3}^n \frac{k^{\alpha\gamma-3}}{N^{\alpha\gamma}}\label{EQ_1550}.
\end{align}
For $\alpha \gamma <2$, obviously $\sum_{k=3}^n k^{\alpha\gamma-3}$ is summable.  Hence,  \eqref{EQ_1550} reduces to
\begin{align}
\sum_{k=3}^n \tau_k^3t_{k-1}^{\alpha-3} \leq CN^{-\alpha\gamma}\label{EQ_<2}.
\end{align}
For $\alpha \gamma = 2$,  it is easy to verify $\sum_{k=3}^nk^{-1}\leq C \ln N$, then we have
\begin{align}
\sum_{k=3}^n \tau_k^3t_{k-1}^{\alpha-3} \leq CN^{-2}\ln N\label{EQ_=2}.
\end{align}
For $\alpha \gamma > 2$, note that $\sum_{k=3}^n \frac1{N}(\frac{k}{N})^{\alpha\gamma-3}\leq \int_0^1 x^{\alpha\gamma-3} \d x\leq \frac{1}{\alpha\gamma-2}$, we arrive at
\begin{align}
\sum_{k=3}^n \tau_k^3t_{k-1}^{\alpha-3} \leq CN^{-2}\label{EQ_>2}.
\end{align}
The proof is completed by inserting \eqref{EQ_<2}, \eqref{EQ_=2} and \eqref{EQ_>2} into \eqref{EQ_1610}.
\end{proof}

\section{Numerical experiments\label{Section5}}
Two numerical examples are reported here to demonstrate our theory.
\begin{example}[Construct an exact solution for abstract PIDE]\end{example}
The IMEX BDF2 scheme \eqref{IMEXsp} runs for solving the problem \eqref{continuous} in the spatial domain $(0,\pi)$ and time interval $(0,1]$. We take $c_1=c_2=c_3=1$, $u_b(x,t)=0$ and $\rho(x)\equiv 1$. The source term $f$ is chosen such that the continuous problem \eqref{continuous} has a exact    solution $u=(1+t^\alpha)\sin(x)$ with some parameters $0.5\leq \alpha\leq 1$.

In our simulations, we uniformly divide the spatial domain $\Omega$  into $M$ subintervals  and the time interval $[0, 1]$  by a graded mesh with $N$ points. Since the spatial error $\mathcal{O}(h^2)$ is standard, we only investigate the temporal error. In each run, fixing $M=8192$, then the discrete $L^2$-error is recorded as $e(N)= \|u(t_n,\cdot)-u^n_h\|$ and the temporal convergence rate (listed as ``Order'' in the tables) is calculated by $\text{Order}=\log_2(e(N)/e(2N))$.

To test the sharpness of our error estimate \eqref{EQ_Graded}, we list the $L^2$-error and convergence order  in Table \ref{table1}, in which  different parameters $\alpha, \gamma$ are chosen. As shown in Table \ref{table1},  the numerical results   support  the predicted time accuracy in Theorem \ref{Theorem_Graded} on the smoothly graded mesh $t_k=T(k/N)^\gamma$. The
nonuniform meshes evidently improve the numerical precision and convergence order of solution. More specifically, one can observe an accuracy of $\mathcal{O}(N^{-\alpha\gamma})$ for $\gamma<2/\alpha$, and also observe the accuracy of $\mathcal{O}(N^{-2})$ for $\gamma>2/\alpha$. Results in Table \ref{table1} suggest that $\gamma=2/\alpha$ is the optimal choice to have the optimal convergence order, which agrees with our theoretical error estimate \eqref{EQ_Graded}.
\begin{table}[!ht]
\begin{center}
\caption {  Errors and convergence orders in temporal direction  } \vspace*{0.5pt}\label{table1}
\def\temptablewidth{1.0\textwidth}
{\rule{\temptablewidth}{1pt}}
\begin{tabular*}{\temptablewidth}{@{\extracolsep{\fill}}cccccccccccc}
$\alpha$ &$N$ &\multicolumn{2}{c}{$\gamma=1$}
&\multicolumn{2}{c}{$\gamma=2$} &\multicolumn{2}{c}{$\gamma=3$}&\multicolumn{2}{c}{$\gamma=4$}\\
\cline{3-4}    \cline{5-6} \cline{7-8}\cline{9-10}
&& $e(N)$ & Order &$e(N)$ & Order &$e(N)$&Order&$e(N)$&Order\\
\hline
&$2^9$  &3.4147e-01 &--&1.1743e-02 &-- &1.1656e-02 &-- &5.0170e-03 &--\\
&$2^{10}$ &2.4085e-01 &0.50 &5.8713e-03 &1.00&4.1598e-03 &1.49 &1.3395e-03 &1.91\\
0.5&$2^{11}$ &1.7010e-01 &0.50 &2.9354e-03 &1.00 &1.4802e-03 &1.49 &3.5599e-04 &1.91\\
&$2^{12}$ &1.2020e-01 &0.50 &1.4676e-03 &1.00 &5.2559e-04 &1.49 &9.4177e-05 &1.92\\
&$2^{13}$  &8.4971e-02 &0.50 &7.3371e-04 &1.00 &1.8632e-04 &1.50 &2.4748e-05 &1.93 \\
\hline
&$\gamma\alpha$&&0.50&&1.00&&1.50&&2.00 \\
\hline
&$2^{9}$ &3.5558e-02 &-- &1.2513e-03 &-- &4.7820e-04 &-- &3.4767e-04 &--\\
&$2^{10}$ &2.0958e-02 &0.76 &4.4735e-04 &1.48 &1.2208e-04 &1.97 &8.5749e-05 &2.02\\
0.75&$2^{11}$ &1.2408e-02 &0.76 &1.5931e-04 &1.49 &3.0952e-05 &1.98 &2.1178e-05 &2.02\\
&$2^{12}$ &7.3626e-03 &0.75 &5.6542e-05 &1.49 &7.7402e-06 &2.00 &5.1632e-06 &2.04\\
&$2^{13}$ &4.3734e-03 &0.75 &1.9973e-05 &1.50 &1.8480e-06 &2.07 &1.1774e-06 &2.13\\
\hline
&$\gamma\alpha$&&0.75&&1.50&&2.00 && 2.00\\
\hline
&$2^{9}$ &5.7505e-03 &--&1.8696e-04 &- &9.1360e-05 &-- &9.6765e-05 &--\\
&$2^{10}$  &2.9891e-03 &0.94 &5.6189e-05 &1.73 &2.2599e-05 &2.02 &2.4009e-05 &2.01\\
0.9&$2^{11}$  &1.5776e-03 &0.92 &1.6683e-05 &1.75 &5.5193e-06 &2.03 &5.8858e-06 &2.03\\
&$2^{12}$  &8.3909e-04 &0.91 &4.8511e-06 &1.78 &1.2691e-06 &2.12 &1.3635e-06 &2.11\\
&$2^{13}$  &4.4805e-04 &0.91 &1.3324e-06 &1.86 &2.1032e-07 &2.59 &2.3444e-07 &2.54 \\
\hline
&$\gamma\alpha$&&0.90&&1.80&&2.00 && 2.00\\
\end{tabular*}
{\rule{\temptablewidth}{1pt}}
\end{center}
\end{table}

\begin{example}[European call option under Merton's model]\end{example}
To simulate a European call option under Merton's model , we only consider the case that $\mu_M$, $\sigma_M$, $\sigma$ and $\lambda$ are constants for simplicity and investigate the convergence orders of IMEX BDF2 scheme with different time grids. By Merton \cite{M76},  the first six terms can obtain the accurate price and the reference values for European call option are 0.52763802 at $S= 90$, 4.39124569 at $S= 100$, and 12.64340583 at $S= 110$.

By choosing the parameters $\sigma = 0.15, r_I = 0.05, \mu_M = -0.9, \sigma_M = 0.45, \lambda = 0.1, c_1 = \sigma^2/2, c_2 = -( r_I - \sigma^2/2 - \lambda\kappa), c_3 = r_I + \lambda, T = 0.25, K = 100, x_l = -1.5, x_r = 1.5$, we demonstrate the errors  at the reference points and convergence orders of IMEX BDF2 scheme with different time grids in Table \ref{call_T}.  Table \ref{call_T} shows that the IMEX BDF2  scheme is convergent with quadratic rate, which is consistent with our theoretical analysis.


\begin{table}[h]
	\centering
	\caption{The errors and convergence orders on graded time mesh with $\alpha=1/2$ and $\gamma=4$ for European call option (Merton's)}
	\begin{tabular}{cccccccc}
		\toprule
		& &S=90& &S=100& &S=110&\\
		$M$ & $N$ & Error & Order & Error & Order & Error & Order\\
		\midrule
		256&256&4.2388e-04& -- &9.0000e-03&-- &2.0000e-03&--\\
		512&512&1.1181e-04&1.92&2.2000e-03&2.01&5.0890e-04&1.99\\
		1024&1024&2.8316e-05&1.98&5.5722e-04&2.00&1.2743e-04&2.00\\
		2048&2048&7.1017e-06&2.00&1.3927e-04&2.00&3.1868e-05&2.00\\
		\bottomrule
	\end{tabular} \label{call_T}
\end{table}

\section{Conclusion}
The  partial integrl-differential equations (PIDEs), which  arises from option pricing theory when the underlying asset follows a jump diffusion process, may suffer from the weak regularity near $t=0$ due to nonsmoothness of the initial value. In this paper, we revisit an implicit-explicit (IMEX) BDF2 \eqref{IMEXsp} with variable time steps for solving the PIDEs with an initial singularity. To handle the initial singularity, nonuniform time steps like graded mesh are an efficient choice.
If the graded mesh is used to deal with the initial singularity \eqref{time_regularity},  one needs to   $r_2=2^{\gamma}-1>15$ when taking $\alpha=0.5$, which significantly breaks the adjacent time-step ratio restriction $r_k <r_{\max} (\approx 4.86)$ in \cite{2019On,2019diffusion,2020reaction}.

To fill the gap,  several novel properties based on the DOC and DCC kernels are developed, including a strictly positive definiteness of the BDF2 kernels $b^{(n)}_{n-k}$  (see Lemma \ref{thpositive}) and a sharper estimate of  DCC kernels $p^{(n)}_{n-k}$   (see Proposition \ref{pttau}) for handling the initial singularity, and a novel convolution-type Young's inequality (see Lemma \ref{Lemma_Dn}) for dealing with the difficulty arising from spatial discretization. We have presented that the IMEX-BDF2 scheme is unconditionally stable and achieves a $\alpha$-order temporal convergence under mild assumptions on the ratio of adjacent time steps \Ass{1}, where $\alpha$ is the regularity of the exact solution (see \ref{time_regularity}). Compared with the previous ratio restrictions in \cite{2019On,2019diffusion,2020reaction}, our ratio restriction has no restriction on the ratio $r_2$. Thus,  an optimal convergence of order $\mathcal{O}(N^{-2})$ is achieved under the graded mesh  $t_k=T(k/N)^{\gamma_{\text{opt}}}$ with an optimal parameter $\gamma_{\text{opt}}=2/\alpha$.

As far as we know, this is the first rigorous proof of second-order convergence for an IMEX-BDF2 scheme for initial singularity problems without severe restrictions on the ratio of adjacent time-steps especially no any restriction on the first ratio $r_2$.

\bibliographystyle{unsrt}
  \bibliography{citation-349457013}

\end{document}